\documentclass[11pt,letterpaper]{article}
\usepackage[utf8]{inputenc}
\usepackage{amsmath,amsthm}
\usepackage{amsfonts}
\usepackage{amssymb,url}
\usepackage{graphicx}

\usepackage[left=2cm,right=2cm,top=2cm,bottom=2cm]{geometry}
\title{Some new maximum VC classes}
\author{Hunter R. Johnson\footnote{John Jay College, CUNY, 899 10th Ave  New York, NY 10019} \thanks{hujohnson@jjay.cuny.edu} }
\newtheorem{lemma}{Lemma}[section] 
\newtheorem{corollary}{Corollary}[section] 
\newtheorem{theorem}{Theorem}[section] 
\newtheorem{proposition}{Proposition}[section] 
\renewcommand{\phi}{\varphi}

\begin{document}
\maketitle
\begin{abstract}
Set systems of finite VC dimension are frequently used in applications relating to machine learning theory and statistics.  Two simple types of VC classes which have been widely studied are the maximum classes (those which are extremal with respect to Sauer's lemma) and so-called Dudley classes, which arise as sets of positivity for linearly parameterized functions.  These two types of VC class were related by Floyd, who gave sufficient conditions for when a Dudley class is maximum.  It is widely known that Floyd's condition applies to positive Euclidean halfspaces and certain other classes, such as sets of positivity for univariate polynomials. 

In this paper we show that Floyd's lemma applies to a wider class of linearly parameterized functions than has been formally recognized to date. In particular we show that, modulo some minor technicalities, the sets of positivity for any linear combination of real analytic functions is maximum on points in general position.  This includes sets of positivity for multivariate polynomials as a special case.

\end{abstract}

\section{Introduction}
Maximum set systems are in some sense the most perfect systems of finite VC dimension.  They arise most notably from the systems given by ``positive" half-spaces in Euclidean space. They also arise as the dual set system associated with a simple arrangement of hyperplanes.  Their desirable features include a certain kind of recursive structure which allows for, among other things, the existence of so-called sample compression schemes, and as such they are central to most approaches to proving the long outstanding sample compression conjecture \cite{BL98,Fl89,FlWa95,BIP,KW}. Some further uses of maximum set systems exist in machine learning and model theory \cite{LBS, GuHi11,J11}. 


In what follows we first provide the definitions for the basic notions of interest, including set systems, VC dimension, the maximum property and linearly parameterized set systems.  We then go on to establish our results in the subsequent section.

Our results relate to two criteria given by Floyd which are sufficient for a linearly parameterized set system to have the maximum property.  While several specific applications of Floyd's theorem have been given, other powerful applications seem to have been overlooked.  In particular, there seems to be no mention in the literature that Floyd's result applies to general multivariate (rather than univariate) polynomials.  More generally we note the important fact that any linear combination of analytic functions satisfies Floyd's criteria.  

\section{Basic definitions}

Let $X$ be a set and $\mathcal{P}(X)$ its power set.  We call any $\mathcal{C} \subseteq \mathcal{P}(X)$ a \textit{set system} on $X$.  For any $X_0 \subseteq X$, we let $\mathcal{C}\vert_{X_0}$ denote $\{C \cap X_0 : C \in \mathcal{C}\}$.  We say that $\mathcal{C}$ \textit{shatters} $X_0 \subseteq X$ if $\mathcal{C}\vert_{X_0} = \mathcal{P}(X_0)$.

The \textit{Vapnik-Chervonenkis (VC) dimension} \cite{VaCh71} of $\mathcal{C}$, when $\mathcal{C}$ is non-empty, is defined as
$$\text{VC}(\mathcal{C}) = \text{sup}\{|X_0| : X_0 \subseteq X \, \text{is shattered by } \mathcal{C}\}.$$

When $\mathcal{C} = \emptyset$ we will use the convention that VC$(\mathcal{C}) = -1$.

If $\text{VC}(\mathcal{C})$ is finite then $\mathcal{C}$ is said to be a \textit{VC-class}. For natural numbers $n$ and $k$, define
\[
{n \choose \leq k} = \begin{cases} \sum_{i=0}^k {n \choose i} &\mbox{if } n \geq k \\
2^k & \mbox{if } n<k. \end{cases}
\]

 A key combinatorial fact relating to VC classes is Sauer's lemma \cite{Sa72,Sh72}.
 
\begin{lemma}
Let $\alpha = \text{VC}(\mathcal{C})$.  Then for any $X_0 \subseteq X$

$$ |\mathcal{C}\vert_{X_0}| \leq {|X_0| \choose \leq \alpha}.$$
\end{lemma}

  We say that $\mathcal{C}$ is \textit{maximum} \cite{We87} of VC dimension $\alpha$ if for any finite $X_0 \subseteq X$
$$ |\mathcal{C}\vert_{X_0}| = {|X_0| \choose \leq \alpha}.$$

Many set systems arise naturally as the family of sets defined by a parameterized formula in a mathematical structure. For instance, let $X$ be a set, $A$ a parameter set, and $f:X\times A \rightarrow \mathbb{R}$ a real-valued function.  We use the notation $f_a:X \rightarrow \mathbb{R}$ to represent the function defined by $x \mapsto f(x,a)$.  Let $pos(f_a) = \{x \in X: f_a(x) > 0 \}$ and $Pos(f) = \{pos(f_a(x)) :a \in A\}$. Then $Pos(f)$ is a set system on $X$ and has a well-defined VC dimension.

An interesting case occurs when $f$ parameterizes a vector space of real-valued functions.  Specifically, suppose that $f_i:X \to \mathbb{R}$ for $i=1,2,\ldots,n$ are linearly independent real-valued functions,  and $f_0:X \rightarrow \mathbb{R}$ is a real-valued function. Let $\mathcal{F} = \{a_1f_1(x) + a_2f_2(x) + \cdots +a_nf_n(x): a_1,\ldots,a_n \in \mathbb{R}\}$, and define $f_0(x) - \mathcal{F}$ to mean $\{f_0(x) - f(x): f \in \mathcal{F}\}.$ Then $\mathcal{F}$ is a real vector space, and $f_0(x) - \mathcal{F}$ is an affine real vector space.  We will use  $Pos(f_0 - \mathcal{F})$ to denote $\{pos(f_0-f): f \in \mathcal{F}\}$.  Set systems of the form $Pos(f_0 - \mathcal{F})$ have been called \textit{Dudley classes} \cite{BL98}.  


The following theorem is due to Dudley \cite{DW,Du99}. Cover  proved a similar (non-affine) result in \cite{Co65}.

\begin{theorem} \label{TTTTT}
If $\mathcal{F}$ is an $n$-dimensional real vector space of real-valued functions defined on $X$, and $f_0:X \rightarrow \mathbb{R}$, then  
$\text{VC}(Pos(f_0-\mathcal{F})) = n$.
\end{theorem}
  
Dudley classes include some natural set systems such as balls in Euclidean space, halfspaces in Euclidean space, and sets of positivity for polynomials, for which the coefficients are regarded as parameters. The following example is due to Dudley \cite{Du79}.
\\

{\bf{Example:}}  We will show that balls in Euclidean 2-space (disks) form a Dudley class. The scheme of the example can be generalized to higher dimensions. Define $f_0(x,y) = -x^2 -y^2$ and $f(x,y) = a_3y+a_2x+a_1$. Then $f_0 - f \in f_0 - \mathcal{F}$ where $\mathcal{F} = \{a_3y+a_2x+a_1:a_1,a_2,a_3 \in \mathbb{R}\}$. Note that $pos(f_0-f)$ describes a disk with center $(\frac{-a_2}{2}, \frac{-a_3}{2})$ and radius $\sqrt{(\frac{a_2}{2})^2+ (\frac{a_3}{2})^2-a_1}$.\footnote{When $(\frac{a_2}{2})^2+ (\frac{a_3}{2})^2-a_1 < 0$, the radius does not exist; in this case $f_0-f$ has no real solutions and $pos(f_0-f)$ describes the empty-set.  Including this degenerate case does not affect the VC dimension, because the empty-set can always be approximated by a sufficiently small disk.} Thus $Pos(f_0 - \mathcal{F})$ is the set system of all disks in the plane. Since it is also a $3$-dimensional Dudley class, we can conclude from Theorem \ref{TTTTT} that the VC dimension of the set of disks in the plane is 3. \qed
\\

The main link between maximum set systems and Dudley classes is due to Floyd \cite{Fl89} (Theorem 8.2). 

\begin{lemma}
Let $\mathcal{F}$ be a vector space of real-valued functions on a set $X$, with $dim(\mathcal{F})=n$.  Let $f_0(x)$ a function on $X$.  Suppose further that
\begin{enumerate}
\item For any $A \subseteq X$ with $|A| = n$, the dimension of $\mathcal{F}$ restricted to $A$ is $n$.
\item For any $f \in \mathcal{F}$, there are at most $n$ zeros of $f_0 - f$ in $X$.
\end{enumerate}
Then $pos(f_0 - \mathcal{F})$ is maximum of VC dimension $n$ on $X$.
\end{lemma}
The assumption (1) above is (by Dudley's theorem) equivalent to the requirement that $pos(f_0 - \mathcal{F})$ shatter every set of size $n$ in $X$, which is a necessary condition for the maximum property. This will be satisfied for an ordinary univariate polynomial $y=a_0 + a_1x + \cdots + a_nx^n$ if $X$ projects 1-1 onto the $x$-axis. 

The assumption (2) above requires that $f_0 \not \in \mathcal{F}$, since otherwise the space $f_0 - \mathcal{F}$ includes the constantly zero function.  If assumption (2) is not respected, Euclidean halfspaces provide a counter-example to the lemma, as observed in \cite{BL98}.

We now proceed to give arguably more natural criteria which guarantee that a linear system of real-valued functions satisfy Floyd's conditions.  For instance if $f_0,\ldots,f_n$ are linearly independent analytic functions and $X \subset \mathbb{R}^k$ is in general position, then with $\mathcal{F} = span \langle f_1,\ldots,f_n \rangle$, Floyd's lemma applies to $f_0 - \mathcal{F}$.

This gives examples of maximum families which have not been given in the literature to date. Some examples are given at the end of the next section.

\section{Results}
In this section we will introduce topological and analytic conditions on $\mathcal{F}$ and $X$ which are sufficient to guarantee that $Pos(f_0 - \mathcal{F})$ is a maximum set system when restricted to subsets of $X$ which are in general position (Theorem \ref{T:dskfjh}). 
The basic strategy for proving Theorem \ref{T:dskfjh} is to associate subsets of $X$ of size $N$ with elements of $X^N$, and observe that the elements not satisfying Floyd's criteria constitute a thin part of $X^N$. 

The elements of $X^N$ on which Floyd's conditions fail will be seen to lie on the zero sets of certain functions arising as determinants of matrices. Establishing that these determinants do indeed have thin zero sets is the aim of Lemma \ref{L:kjh}.

 When $\mathcal{F}$ consists of analytic functions (which are defined before Proposition \ref{P:kdfh}) the elements of $X^N$ not satisfying Floyd's criteria will actually have Lebesgue measure zero.  This implies that if a finite $X_0 \subseteq X$ is selected according to one of several common probability distributions, including the uniform and Gaussian distributions, then $Pos(f_0 - \mathcal{F})$ will almost surely be maximum when restricted to $X_0$ (see Corollary \ref{C:dskfjh}).

\subsection{}

Let $\mathcal{F}$ be an $n$-dimensional vector space of real-valued functions on a topological space $X$.  We will say that $\mathcal{F}$ is \textit{admissible} if for any $f\in \mathcal{F}$, 

\begin{enumerate}

\item $f$ is continuous
\item If $f^{-1}(0)$ has non-empty interior, then $f$ is constantly zero.  
\end{enumerate}
Note that any subspace of an admissible $\mathcal{F}$ is admissible.

Equip $X^n = \underbrace{X \times X \times \cdots \times X}_{n\, \text{times}}$ with the product topology.
\begin{lemma}\label{L:kjh}
Suppose $\mathcal{F}$ is admissible and $f_1(x),\ldots,f_n(x)$ is a basis for $\mathcal{F}$. Let $F:X^n \to \mathbb{R}$ be given by

\[F(x_1,\ldots,x_n) = det
 \begin{pmatrix}
  f_{1}({x}_1) & f_{2}({x}_1) & \cdots & f_{n}({x}_1) \\
  f_{1}({x}_2) & f_{2}({x}_2) & \cdots & f_{n}({x}_2) \\
  \vdots  & \vdots  & \ddots & \vdots  \\
  f_{1}({x}_n) & f_{2}({x}_n) & \cdots & f_{n}({x}_n)
 \end{pmatrix}
\]
then $F^{-1}(0) \subseteq X^n$ has empty interior.
\end{lemma}
\begin{proof}
The argument is by induction on $n$.  If $n=1$, the lemma follows from the assumptions on $\mathcal{F}$.

Suppose the lemma is known to hold for $n-1$ and consider $F(x_1,\ldots,x_n)$.  Then

\[F = f_1(x_1) \begin{vmatrix}
  f_{2}({x}_2) & \cdots & f_{n}({x}_2) \\
  \vdots  & \ddots & \vdots  \\
  f_{2}({x}_n) & \cdots & f_{n}({x}_n)
 \end{vmatrix} + \cdots + f_n(x_1)(-1)^{1+n} \begin{vmatrix}
  f_{1}({x}_2) & \cdots & f_{n-1}({x}_2) \\
  \vdots  & \ddots & \vdots  \\
  f_{1}({x}_n) & \cdots & f_{n-1}({x}_n)
 \end{vmatrix}
\]
where the vertical bars denote the determinant.

Suppose $U \subset F^{-1}(0)$ is open.  Assume, by way of contradiction, that $U$ is nonempty.  Let $V$ be the projection of $U$ onto $x_2,\ldots,x_n$.  By inductive hypothesis, there is some $(a_2,\ldots,a_n) \in V$ such that 

\[
\begin{vmatrix}
  f_{2}({a}_2) & \cdots & f_{n}({a}_2) \\
  \vdots  & \ddots & \vdots  \\
  f_{2}({a}_n) & \cdots & f_{n}({a}_n)
 \end{vmatrix} \neq 0
 \]
 This gives
 $$F(x_1,a_2,\ldots,a_n) = c_1 f_1(x_1) + \cdots + c_n f_n(x_1)$$ 
 for real numbers $c_1,\ldots,c_n$, corresponding to the subdeterminants, and $c_1 \neq 0$.  Define $U_{a_2,\ldots,a_n} = \{a \in X : (a,a_2,\ldots,a_n) \in U\}$.  Then $U_{a_2,\ldots,a_n}$ is non-empty and open, and on this open set $F(x_1,a_2,\ldots,a_n) = 0$.  But $F(x_1,a_2,\ldots,a_n) \in \mathcal{F}$, and therefore $F(x_1,a_2,\ldots,a_n) = 0$ everywhere.  This contradicts the linear independence of $f_1,\ldots,f_n$, because $c_1 \neq 0$.  Thus $F^{-1}(0)$ has empty interior.
 
\end{proof}

Proposition \ref{P:kdfh}, below, appears in \cite{Fe96} on p. 240.  The statement given there is for the more general context of Banach spaces.  The proof uses the technique of approximate differentiation; we will give a more elementary argument.

Recall that an infinitely differentiable function $f: \mathbb{R}^n \rightarrow \mathbb{R}$ is \textit{analytic} if for every $x$ in the domain of $f$ there is an open set $U$ with $x\in U$ such that $f$ is equal to its Taylor series expansion on $U$.  
We will use the fact that if $f: \mathbb{R} \rightarrow \mathbb{R}$ is analytic and takes non-zero values, then its zeros form a countable set \cite{KP02}.

\begin{proposition} \label{P:kdfh}
Let $f:\mathbb{R}^k \rightarrow \mathbb{R}$ be analytic.  Suppose $f$ is not constantly zero and let $A= f^{-1}(0)$. Then $\lambda(A)=0$ where $\lambda$ is Lebesgue measure.
\end{proposition}
\begin{proof}
 The proof is by induction on $k$.  Suppose $k=1$.  Then $A$ is countable and therefore $\lambda(A)= 0$.  If $k>1$ then let $\chi_A$ be the indicator function for $A$. 
 


Define $A_r = A \cap B_r$ where $B_r$ is a $k$-ball of radius $r$ centered at the origin. Then the functions $\{\chi_{A_n} : n \in \mathbb{N}\}$ converge monotonically to $\chi_A$.  By the monotone convergence theorem \cite{Ro88}, 
$\lim_{n \to \infty} \int \chi_{A_n} = \int \chi_A$.  Thus it suffices to show that $\int \chi_{A_n} = 0$ for all $n$. By replacing $A$ with $A_n$ if necessary, we may assume without loss that $\lambda(A) < \infty$.

Since $\chi_A$ takes only non-negative values, $\int_{\mathbb{R}^k} |\chi_A| d(x_1,\ldots,,d_k) = \int_{\mathbb{R}^k} \chi_A d(x_1,\ldots,,d_k) = \lambda(A) < \infty$.  Thus $\chi_A$ satisfies the conditions of Fubini's theorem \cite{Fe96}.  By Fubini's theorem  we may evaluate $\int_{\mathbb{R}^k} \chi_A \,d(x_1,\ldots,x_k)$ by iterated integration. 
  From iterated integration and induction it is seen that $\int_{\mathbb{R}^k} \chi_A \,d(x_1,\ldots,x_k) = 0$.  This implies that $\lambda(A) = 0$.
 \end{proof}

\begin{corollary}\label{P:dkjh}
 If $\mathcal{F}$ is a real vector space of real-valued functions defined on  $\mathbb{R}^k$ and $\mathcal{F}$ has a basis consisting of real analytic functions then $\mathcal{F}$ is admissible. 
\end{corollary}
\begin{proof}
Lebesgue measure zero implies empty interior.
\end{proof}

\begin{corollary} \label{C:fjkh}
Suppose that $\mathcal{F}$ is a real vector space of real-valued functions defined on $\mathbb{R}^k$ for some $k \in \mathbb{N}$ and $\mathcal{F}$ has a basis of real analytic functions. Then $F^{-1}(0)$ has Lebesgue measure zero, where $F$ is as in Lemma \ref{L:kjh}.
\end{corollary}
\begin{proof}
Note that $F$ is analytic and not constantly zero, and therefore Proposition \ref{P:kdfh} applies. 
\end{proof}

We will call a non-empty topological space $X$ a \textit{Baire space} if any countable union of closed sets with empty interior has empty interior.  For natural numbers $n$ and $N$, by $[N]^n$ we mean the subsets of $\{1,\ldots,N\}$ of cardinality $n$.  We will abuse notation slightly by writing $\langle i_1,\ldots,i_n \rangle \in [N]^n$ to mean that $\{i_1,\ldots,i_n\} \in [N]^n$ and $i_1<\cdots <i_n$. For an ordered set $q \in X^N$, we regard $q$ as the function with domain $[N]$ and codomain $X$ defined by $i \mapsto q_i$.  Thus by $range(q)$ we mean the elements of $X$ occurring in the ordered set $q$.

\begin{theorem}\label{T:dskfjh}
Suppose $f_0,f_1,\ldots,f_n$ are linearly independent real-valued functions defined on an infinite topological space $X$ with the property that for every $N \in \mathbb{N}$, $X^N$ is Baire in the product topology. Put $\mathcal{F} = span \langle f_0,f_1,\ldots,f_n \rangle$ and $\mathcal{C} = pos(f_0 - span\langle f_1,\ldots,f_n\rangle )$.  Then if $\mathcal{F}$ is admissible then for every $N > n$ there is $X_0 \subseteq X$ with $|X_0|=N$ such that $\mathcal{C} \vert_{X_0}$ is maximum of VC dimension $n$.
\end{theorem}
\begin{proof}
Let $N \in \mathbb{N}$ be given.  Let $x_1,\ldots,x_N$ be variables ranging over $X$.  We can express the statement that $x_1,\ldots,x_N$ satisfy conditions (1) and (2) of Floyd's lemma using determinants. 
For any $B = \langle i_1,\ldots,i_n \rangle \in [N]^n$, let $F_B$ denote the function 
 
\[
F_B(x_{1},\ldots,x_{N}) = det
 \begin{pmatrix}
  f_{1}({x}_{i_1}) & f_{2}({x}_{i_1}) & \cdots & f_{n}({x}_{i_1}) \\
  f_{1}({x}_{i_2}) & f_{2}({x}_{i_2}) & \cdots & f_{n}({x}_{i_2}) \\
  \vdots  & \vdots  & \ddots & \vdots  \\
  f_{1}({x}_{i_n}) & f_{2}({x}_{i_n}) & \cdots & f_{n}({x}_{i_n})
 \end{pmatrix}.
\]
  Note that $F_B$ ignores variables not in $B$.
  
   Condition (1) will be true if for every $B \in [N]^n$, $F_B(x_1,\ldots,x_N) \neq 0$.    Note that for each choice of $B \in [N]^n$, $F_B^{-1}(0)$ has empty interior as a subset of $X^N$, as a consequence of Lemma \ref{L:kjh}.

Condition (2) of Floyd's lemma will be satisfied if the system given by

\[
 \begin{pmatrix}
  f_{1}({x}_{i_1}) & f_{2}({x}_{i_1}) & \cdots & f_{n}({x}_{i_1}) \\
  f_{1}({x}_{i_2}) & f_{2}({x}_{i_2}) & \cdots & f_{n}({x}_{i_2}) \\
  \vdots  & \vdots  & \ddots & \vdots  \\
  f_{1}({x}_{i_{n+1}}) & f_{2}({x}_{i_{n+1}}) & \cdots & f_{n}({x}_{i_{n+1}})
 \end{pmatrix}
\begin{pmatrix}
 w_1 \\ w_2 \\ \vdots \\ w_n
\end{pmatrix}
=
\begin{pmatrix}
 f_0({x}_{i_1}) \\ f_0({x}_{i_2}) \\ \vdots \\ f_0({x}_{i_{n+1}})
\end{pmatrix}
\]
is inconsistent for every choice of $x_{i_1},\ldots,x_{i_{n+1}}$ from $x_1,\ldots,x_N$.
If we define, for every $B = \langle i_1,\ldots,i_{n+1} \rangle \in [N]^{n+1}$, 
\[
G_B(x_{1},\ldots,x_{N}) = det
 \begin{pmatrix}
  f_{1}({x}_{i_1}) & f_{2}({x}_{i_1}) & \cdots & f_{n}({x}_{i_1}) & f_0({x}_{i_1}) \\
  f_{1}({x}_{i_2}) & f_{2}({x}_{i_2}) & \cdots & f_{n}({x}_{i_2}) & f_0({x}_{i_2}) \\
  \vdots  & \vdots  & \ddots & \vdots & \vdots \\
  f_{1}({x}_{i_{n+1}}) & f_{2}({x}_{i_{n+1}}) & \cdots & f_{n}({x}_{i_{n+1}}) &f_0({x}_{i_{n+1}})
 \end{pmatrix}
\]
then condition (2) is equivalent to the requirement that $G_B(x_1,\ldots,x_N) \neq 0$ for all $B \in [N]^{n+1}$.  Note that for each choice of $B \in [N]^{n+1}$, $G_B^{-1}(0)$ has empty interior as a subset of $X^N$, as a consequence of Lemma \ref{L:kjh}.

To complete the argument, we must show that

$$ Q := X^N \setminus \left ( \bigcup_{B \in [N]^n} F_B^{-1}(0) \cup \bigcup_{B \in [N]^{n+1}} G_B^{-1}(0) \right ) $$
is nonempty.  But since $X^N$ is Baire, $Q$ is actually open and dense.  Taking $q \in Q \subseteq X^N$, we see that $X_0 := range(q)$ suffices, by Floyd's lemma.
\end{proof}

Note that in the special case in which $f_0,\ldots,f_n$ are real analytic functions defined on $\mathbb{R}^k$, the set $Q$ as in the proof of the theorem is not only dense and open, but co-null in the sense of Lebesgue measure by Corollary \ref{C:fjkh}.  This gives applications to probability distributions which have the same null sets as Lebesgue measure. Recall that a measure $\nu$ defined on the Borel sets is \textit{absolutely continuous} with respect to $\lambda$ if $\lambda(B) =0$ always implies $\nu(B) = 0$.  It is known that the Gaussian measures are absolutely continuous with respect to Lebesgue measure \cite{KS06}.  The same is true for the uniform probability measure defined on a box in Euclidean space, since this is just Lebesgue measure normalized to a bounded set.  

\begin{corollary} \label{C:dskfjh}
Suppose $f_0,f_1,\ldots,f_n$ are linearly independent real analytic functions defined on $$R = \underbrace{[0,1] \times [0,1] \times \cdots \times [0,1]}_{k\, \text{times}}.$$ Put $\mathcal{C} = pos(f_0 - span\langle f_1,\ldots,f_n\rangle )$. For $N > n$, let $\nu$ be a probability measure on $R^N$ which is absolutely continuous with respect to Lebesgue measure.  Then if $q \in R^N$ is selected at random according to $\nu$, then $\mathcal{C} \vert_{X_0}$ is maximum of VC dimension $n$ with probability 1, where $X_0 = range(q)$.
\end{corollary}
\begin{proof}
Observe that $q \in Q$ almost surely, because $\nu(Q) = \lambda(Q) = 1$.
\end{proof}


Note that for any set of real variables $V=\{v_1,\ldots,v_m\}$, distinct monomials arising from $V$ are linearly independent and analytic.  Thus the above results apply, in particular, to polynomial functions and their sets of positivity. 

This generalizes the Floyd/Dudley result which states that the set of open balls in a Euclidean space has the maximum property on points in general position.

It also applies to some functions which seem not to have been considered before, such as trigonometric polynomials.  That is, functions of the form

$$t(x,y;a_0,a_1,\ldots,a_N,b_1,\ldots,b_N) = a_0 + \sum_{n=1}^N a_n\cos(nx) + \sum_{n=1}^N b_n\sin(nx) - y$$
where the $a_i$ and $b_i$ are viewed as parameters.  The Wronskian criterion \cite{So01} for the linear independence of functions can be used to generate still more examples.

The notion of samples which are dense (in the product topology) with certain properties has been undertaken by Sontag in the context of neural networks \cite{S95}.

\end{document}